\documentclass[12pt]{amsart}
\usepackage{amssymb,amsmath,mathrsfs}
\usepackage[ps2pdf,colorlinks=true,urlcolor=blue,
citecolor=red,linkcolor=blue,linktocpage,pdfpagelabels,bookmarksnumbered,bookmarksopen]{hyperref}
\usepackage{epsfig,graphicx,color,mathrsfs}
\usepackage[english]{babel}
\usepackage[left=3cm,right=3cm,top=2.45cm,bottom=2.45cm]{geometry}

\numberwithin{equation}{section}
\newcommand{\dvg}{{\rm div}\,}

\newcommand{\Z}{{\mathbb Z}}
\newcommand{\R}{{\mathbb R}}
\newcommand{\Dis}{{\mathbb D}}
\newcommand{\Sf}{{\mathbb S}}

\newcommand{\eps}{\varepsilon}

\newcommand{\dom}[1]{{\rm dom}(#1)}
\newcommand{\ws}[1]{|d#1|}
\newcommand{\wsr}[1]{|d_{\rho}#1|}
\newcommand{\gra}[1]{{\mathcal G}_{#1}}

\newcommand{\ball}[2]{B_{#2}\left(#1\right)}
\newcommand{\epi}[1]{{\rm epi}\left(#1\right)}

\newcommand{\refe}[1]{{(\ref{#1})}}
\newcommand{\dys}{\displaystyle}
\newcommand{\into}{{\int_{B_1}}}
\newcommand{\hsob}{H^1_0(B_1)}

\newcommand{\elle}[1]{L^{#1}(B_1)}
\newcommand{\js}[1]{j_s({#1},|\nabla {#1}|)}
\newcommand{\jxi}[1]{j_\xi({#1},|\nabla {#1}|)}

\newcommand{\cal}{\mathcal}

\numberwithin{equation}{section}
\newtheorem{theorem}{Theorem}[section]
\newtheorem{proposition}[theorem]{Proposition}
\newtheorem{lemma}[theorem]{Lemma}
\newtheorem{remark}[theorem]{Remark}

\newtheorem{definition}[theorem]{Definition}
\theoremstyle{definition}

\newcommand{\edm}{\end{displaymath}}
\numberwithin{equation}{section}

\newcommand{\bos}{\begin{remark}\rm}

\newcommand{\ben}{\begin{enumerate}}
\newcommand{\een}{\end{enumerate}}

\newcommand{\be}{\begin{equation}}
\newcommand{\ee}{\end{equation}}

\title[Symmetric minimax critical points 
for nonsmooth functionals]{Radial symmetry of minimax critical \\
points for nonsmooth functionals}

\author{Marco Squassina}

\address{Dipartimento di Informatica
\newline\indent
Universit\`a degli Studi di Verona
\newline\indent
C\'a Vignal 2, Strada Le Grazie 15, I-37134 Verona, Italy}
\email{marco.squassina@univr.it}

\thanks{The author was partially supported by the
Italian PRIN Research Project 2007: {\em Metodi Variazionali e Topologici
nello Studio di Fenomeni non Lineari}}

\begin{document}

\subjclass[2000]{35J40; 58E05}

\keywords{Nonsmooth critical point theory, minimax theorems, polarization,
abstract symmetrization, Schwarz symmetrization, quasi-linear elliptic equations}

\begin{abstract}
We obtain the existence of radially symmetric and decreasing
solutions to a general class of quasi-linear
elliptic problems by a nonsmooth version of a symmetric minimax principle 
recently obtained by Jean Van Schaftingen.
\end{abstract}
\maketitle

\section{Introduction and main result}

\subsection{Introduction}
The main goal of this paper is to provide, in the framework of nonsmooth
critical point theory (cf.~\cite{corv,cdm,dm} and references therein), a general minimax variational
principle for a class of lower semi-continuous functionals satisfying certain
monotonicity properties under polarization, allowing to detect
critical points in the sense of the weak slope (cf.~Definitions~\ref{defslope} and~\ref{defnwslsc})
of minimax type which are radially symmetric and
decreasing. In the case of $C^1$ smooth functionals
these type of results were studied by Jean Van Schaftingen in~\cite{jvsh} (see also~\cite{jvsh1,jvsh2}),
where various applications to semi-linear elliptic equations of the form
$-\Delta u=g(|x|,u)$ in $\Omega$
with $u=0$ on $\partial\Omega$ were also derived under suitable assumptions
on $g$, when $\Omega$ is either a ball $B_1$ or an annulus (see also~\cite{jvshW}). 
On the other hand, typically, in the general context
of quasi-linear problems of variational type, the energy functional 
$f:H^1_0(B_1)\to\R\cup \{+\infty\}$ is, say,
\begin{equation}
	\label{quasi-ff}
f(u)=\int_{B_1}j(u,|\nabla u|)dx
-\int_{B_1}G(|x|,u)dx,
\end{equation}
and under reasonable assumptions $f$ is merely either lower semi-continuous 
or continuous on $H^1_0(B_1)$, depending on the
growth conditions which are imposed on $j$ and $G$ 
(cf.~\cite{canino,pelsqu,toulouse}). A class of minimization problems, 
constrained to the unit sphere of $L^p(\R^N)$, for functionals~\eqref{quasi-ff} defined on the whole $\R^N$ has been recently
investigated in~\cite{HSq} by exploiting the following generalized Polya-Szeg\"o
and Hardy-Littlewood type inequalities for Schwarz symmetrization
\begin{equation*}
\int_{\R^N}j(u^*,|\nabla u^*|)dx
\leq\int_{\R^N}j(u,|\nabla u|)dx,\quad
\int_{\R^N}G(|x|,u)dx \leq\int_{\R^N}G(|x|,u^*)dx,
\end{equation*}
the latter holding true under suitable monotonicity conditions 
on $G$ in the radial argument. These inequalities also hold in the unit ball $B_1$ and immediately yield 
$f(u^*)\leq f(u)$ for all $u\in H^1_0(B_1)$, namely~\eqref{quasi-ff} decreases upon Schwarz
symmetrization. In turn, the existence of a global minimizer for $f$ on a sphere 
$\{u\in H^1_0(B_1):\|u\|_{L^p}=1\}$, with $p>1$, immediately yields
the existence of a radially symmetric and decreasing minimizer.
The first of the previous symmetrization inequalities holds under mild assumptions,
allowing $j(s,|\xi|)$ to be unbounded with respect to $s$, say, for instance 
$j(s,|\xi|)\leq\alpha(|s|)|\xi|^p$ where $\alpha:\R^+\to\R^+$ is an increasing function.
This constrained minimization
problems arise, for instance, in the study of standing wave solutions for semi-linear
and quasi-linear Schr\"odinger equations (see~\cite{CJS} for a recent study).
Concerning the study of free critical points for $f$, in~\cite{pelsqu} it was
obtained existence of infinitely many critical points via $\Z_2$-symmetric nonsmooth mountain
pass theorems, under (a subset of) the assumptions
listed in Section~\ref{mainres} (see also~\cite{sq-monograph} and references therein
for various applications of nonsmooth critical point theory to quasi-linear elliptic problems). 
In this paper, we shall prove a general nonsmooth
minimax principle (cf.~Theorem~\ref{mpconc}) and, in turn, we shall derive the main abstract 
result of the paper (cf.~Theorem~\ref{mpconc-symm}), a 
symmetric version of Theorem~\ref{mpconc} working
for a large class of lower semi-continuous functionals 
(in abstract spaces) which are decreasing upon (abstract) polarization. 
In our main concrete result (cf.~Theorem~\ref{mainth}) we state
the existence of a nontrivial radially symmetric and decreasing distributional solution of problem 
\begin{equation}
	\label{theproblem}
	\tag{$P$}
\begin{cases}
-\,\dvg(j_\xi(u,|\nabla u|)) +j_s(u,|\nabla u|)=g(|x|,u), & \text{in $B_1$}, \\
\quad u=0, & \text{on $\partial B_1$},
\end{cases}
\end{equation}
corresponding to the {\em mountain pass critical level} of the functional $f$. 
The weak slope critical points $u$ of $f$ naturally 
correspond to generalized solutions (see Definition~\ref{defsol}) of problem~\eqref{theproblem},
which become in turn distributional by showing that $u$ is bounded.
We point out that Theorem~\ref{mpconc-symm} often provides, in general contexts, also an alternative tool
to concentration compactness arguments, see Remark~\ref{conccomp-r} for more details.
Even in the classical cases such as $j(s,\xi)=|\xi|^2$, if the nonlinearity $g(|x|,s)=D_sG(|x|,s)$ is a merely
continuous function, the moving plane argument (cf.~\cite{gnn}) and the homotopy
argument due to Brock (cf.~\cite{brock}) yielding local symmetry of positive solutions
cannot be applied. In the general quasi-linear setting,
even for functions $g$ of class $C^1$, to the author's knowledge, no symmetry results
based upon moving plane arguments are available in the current literature. On the contrary, for the $p$-Laplacian
operator $j(|\xi|)=|\xi|^p$, there are various results for positive solutions 
and ${\rm Lip}_{{\rm loc}}$ and autonomous 
nonlinearities (cf.~\cite{brock1,Dam,DS,DS1} for equations and~\cite{mss} for systems).
Finally we notice that, in some cases, the symmetry can be inferred
by Palais's {\em symmetric criticality principle} (cf.~\cite{palais}) restricting
the functional to radial functions. Of course,
in this case, one would loose the global mountain pass minimization property.

\subsection{The main concrete result}
\label{mainres}

Let $B_1$ be the unit ball in $\R^N$ centered at the origin, $N\geq 3$ and let 
$f:H^1_0(B_1)\to \R\cup\{+\infty\}$ be the functional defined by
\begin{equation*}
f(u)=\int_{B_1}j(u,|\nabla u|)dx
-\int_{B_1}G(|x|,u)dx,
\end{equation*}
where $j(s,|\xi|):\R\times\R^+\to\R$ is of class $C^1$. We consider the following assumptions.

\subsubsection{Assumptions on $j$}
We assume that for every $s$ in $\R$
\begin{equation}\label{j1}
\text{$\big\{|\xi|\mapsto j(s,|\xi|)\big\}$
is strictly convex and increasing}.
\end{equation}
Moreover, there exist a constant $\alpha_0>0$ and
a positive increasing function $\alpha\in C(\R)$
such that, for every $(s,\xi)\in\R\times\R^N$, it holds
\begin{equation}\label{j2}
\alpha_0|\xi|^2\leq j(s,|\xi|)\leq \alpha(|s|)|\xi|^2.
\end{equation}
The functions $j_s(s,|\xi|)$ and $j_\xi (s,\xi)$
denote the derivatives of $j(s,\xi)$ with
respect of the variables $s$ and $\xi$
respectively.
Regarding the function $j_s(s,|\xi|)$,
we assume that there exist a positive increasing function
$\beta\in C(\R)$ and a positive constant $R$ such that 
\begin{equation}\label{j3}
\big|j_s(s,|\xi|)\big|\leq \beta(|s|)|\xi|^2,\qquad \text{
for every $s$ in $\R$ and all $\xi\in\R^N$},
\end{equation}
\begin{equation}\label{j4}
\qquad \qquad  j_s(s,|\xi|)s\geq 0,
\qquad \text{for every $s$ in $\R$ with $|s|\geq R$ and all $\xi\in\R^N$.}
\end{equation}
Furthermore, we assume that
\begin{equation}
	\label{opposite}
\qquad \qquad j(-s,|\xi|)\leq j(s,|\xi|),
\qquad\text{for every $s$ in $\R^-$ and all $\xi\in\R^N$}.
\end{equation}
\subsubsection{Assumptions on $g,G$}
The function $G(|x|,s)$ is the primitive with respect to $s$ of a Carath\'eodory function $g(|x|,s)$
such that $G(|x|,0)=0$. We assume that there exist $p\in (2,2N/(N-2))$, 
a positive constant $C$, $\mu>2$ and $R'>0$ such that
\begin{equation}
\label{gg1}
|g(|x|,s)|\leq C(1+|s|^{p-1}),\qquad\text{for every $s$ in $\R$ and $x\in B_1$,}
\end{equation}
\begin{equation}
\label{gg2}
0<\mu G(|x|,s)\leq g(|x|,s)s,\qquad\text{for every $s$ in $\R$ with $|s|\geq R'$ and $x\in B_1$,}
\end{equation}
\begin{equation}
\label{gg3}
\lim_{s\to 0}  \frac{g(|x|,s)}{s}=0,\quad\text{uniformly in $B_1$},
\end{equation}
\begin{equation}
\label{gg5}
g(|x|,s)\geq g(|y|,s),\qquad\text{for every $s\in\R$ and $x,y\in B_1$ with $|x|\leq |y|$},
\end{equation}
\begin{equation}
\label{gg6}
G(|x|,s)\leq G(|x|,-s),\qquad\text{for every $s\in\R^-$ and $x\in B_1$}.
\end{equation}
\subsubsection{Joint assumptions of $j$ and $g$}
There exist $R''>0$ and $\delta>0$ such that
\begin{equation}
\label{j5}
pj(s,|\xi|)-j_s(s,|\xi|)s-j_\xi(s,|\xi|)\cdot\xi\geq \delta|\xi|^2,
\quad\text{for every $s\in\R$ with $|s|\geq R''$}
\end{equation}
and all $\xi\in\R^N$. Finally, it holds 
\begin{equation}
\label{j6}
\lim\limits_{|s|\to\infty} \frac{\alpha(|s|)}{|s|^{p-2}}=0.
\end{equation}

\begin{remark}\rm
	The asymptotic sign condition~\eqref{j4} is typical for quasi-linear elliptic problems 
	and, in general, plays a r\^ole both in the regularity
	theory (see Frehse's counterexample in~\cite{frese}) and in the verification of the Palais-Smale 
	condition (see e.g.~\cite{pelsqu,toulouse}). Assumption~\eqref{gg5} is necessary
	in order to get some inequality for $\int G(|x|,u)$ under polarization, while~\eqref{j5} is used
	to prove that Palais-Smale sequences are bounded (see~\cite{pelsqu,toulouse}). 
	Assumption~\eqref{j6} is needed for the functional
	to satisfy some Mountain Pass geometry (see~\cite{pelsqu}).
	Finally, we point out that the growth~\eqref{gg1} on $g$ could be weakened, allowing
	that for all $\eps>0$ there exists $a_\eps\in L^r(B_1)$ with $r>\frac{2N}{N+2}$ such that $|g(|x|,s)|\leq a_\eps(x)+\eps|s|^{(N+2)/(N-2)}$ for all $x\in B_1$ and $s\in\R$. If $r>N/2$ the solutions
	are bounded (cf.~\cite[Theorem 7.1(b)]{pelsqu}).
\end{remark}

\subsubsection{Statement} 

The principal result of the paper is the following general
Ambrosetti-Rabinowitz~\cite{ambrab} mountain type theorem
which includes the additional information on the radial symmetry
of the solution.

\begin{theorem}
\label{mainth}
 Assume that~\eqref{j1}-\eqref{j6} hold. Then there exists a nontrivial
radially symmetric and decreasing mountain pass distributional
solution $u\in H^1_0\cap L^\infty(B_1)$ to~\eqref{theproblem}. 
\end{theorem}

Although we state our result in the unit ball, a 
similar statement could be provided for the functional defined 
in the unit ball or in the annulus without the monotonicity condition~\eqref{gg5} on $g$,
yielding a distributional solution $u$ which is invariant under spherical cap symmetrization,
namely $u$ depends solely upon $|x|$ and an angular variable.
In~\cite{jvsh} further applications of the symmetric minimax principle
for $C^1$ smooth functionals are provided, for instance the case of linking geometry. We limit ourself to the statement
of Theorem~\ref{mainth} although also in the nonsmooth setting some of 
the applications discussed in~\cite{jvsh} could be derived from the symmetric principle,
Theorem~\ref{mpconc-symm}.

\smallskip
\section{Tools from symmetrization theory}

\subsection{Abstract symmetrization}
We recall a definition from~\cite{jvsh}.
\vskip1pt
\noindent
Let $X$ and $V$ be two Banach spaces and $S\subset X$.
We consider two maps $*:S\to V$, $u\mapsto u^*$ 
({\em symmetrization map}) and $h:S\times {\mathcal H}_*\to S$,
$(u,H)\mapsto u^H$ ({\em polarization map}), where ${\mathcal H}_*$ 
is a path-connected topological space. We assume that the following 
conditions hold:
\begin{enumerate}
 \item $X$ is continuously embedded in $V$;
 \item $h$ is a continuous mapping;
\item for each $u\in S$ and $H\in {\mathcal H}_*$ it holds $(u^*)^H=(u^H)^*=u^*$ and $u^{HH}=u^H$;
\item there exists a sequence $(H_m)$ in ${\mathcal H}_*$ such that $u^{H_1\cdots H_m}$ converges
to $u^*$ in $V$;
\item for every $u,v\in S$ and $H\in {\mathcal H}_*$ it holds
$\|u^H-v^H\|_V\leq \|u-v\|_V$.
\end{enumerate}

\subsubsection{Polarization}
A subset $H$ of $\R^N$ is called a polarizer if it is a closed affine half-space
of $\R^N$, namely the set of points $x$ which satisfy $\alpha\cdot x\leq \beta$
for some $\alpha\in \R^N$ and $\beta\in\R$ with $|\alpha|=1$. The family of polarizers
can be compactified by adding two polarizers $H_{+\infty}$ and $H_{-\infty}$ such that
$H_m\to H_{+\infty}$ if $\beta_m\to+\infty$ and $H_m\to H_{-\infty}$ if $\beta_m\to-\infty$.
Given $x$ in $\R^N$
and a polarizer $H$ the reflection of $x$ with respect to the boundary of $H$ is
denoted by $x_H$. The polarization of a function $u:\R^N\to\R^+$ by a polarizer $H$
is the function $u^H:\R^N\to\R^+$ defined by
\begin{equation}
 \label{polarizationdef}
u^H(x)=
\begin{cases}
 \max\{u(x),u(x_H)\}, & \text{if $x\in H$} \\
 \min\{u(x),u(x_H)\}, & \text{if $x\in \R^N\setminus H$.} \\
\end{cases}
\end{equation}
The polarization $u^H$ of a function defined on $C\subset \R^N$
is just the restriction to $C$ of the polarization of the extension $\tilde u:\R^N\to\R$ of 
$u$ to zero outside $C$. The polarization of a function which may change sign is defined
by $u^H:=|u|^H$, for any given polarizer $H$.

\subsubsection{Schwarz symmetrization}
The Schwarz symmetrization of a set $C\subset \R^N$ is the unique open ball $C^*$
such that ${\mathcal L}^N(C^*)={\mathcal L}^N(C)$, being ${\mathcal L}^N$ the
$N$-dimensional Lebesgue measure. If the measure of $C$ is zero, then we set $C^*=\emptyset$,
while if the measure of $C$ is not finite, we put $C^*=\R^N$. The Schwarz symmetrization
of a measurable function $u:C\to\R^+$ is the unique function $u^*:C^*\to\R^+$ such that,
for all $t\in\R$, it holds $\{u^*>t\}=\{u>t\}^*.$
A function is admissible for the Schwarz symmetrization if it is nonnegative and, for every $\eps>0$,
the Lebesgue measure of $\{u>\eps\}$ is finite. Let us set
$$
\Omega=B(0,R)\subset\R^N,
\quad
X=W^{1,p}_0(\Omega),
\quad
S=W^{1,p}_0(\Omega,\R^+),
\quad
V=L^p\cap L^{p^*}(\Omega),
$$
and ${\mathcal H}_*=\{H\in{\mathcal H}: \text{$0\in H$ or $H=H_{+\infty}$}\}$. Then the 
requirements (1),(2),(3),(4),(5) in abstract symmetrization framework are satisfied
by~\cite[Proposition 2.3, Theorem 2.1, Proposition 2.5]{jvsh}. 
Given a function $u:C\to\R$ and considering the extension 
$\tilde u:\R^N\to\R$ of $u$ to zero outside $C$, 
we have $(\tilde u)^*|_{C^*}=u^*$ and $(\tilde u)^*|_{\R^N\setminus C^*}=0$. The symmetrization
for $u$ which are not nonnegative can be the defined by $u^\star:=|u|^*$. In this case we set
$\Omega=B(0,R)$, $X=S=W^{1,p}_0(\Omega)$ and $V=L^p\cap L^{p^*}(\Omega)$.

\begin{remark}\rm
Different types of symmetrization can be considered, such as the Steiner 
symmetrization or the spherical cap symmetrization, for which the abstract framework 
above is fulfilled and our main result would work. We refer the interested reader
to~\cite{jvsh} and to the references therein for further details. See also~\cite{jvsh1,jvsh2}.
\end{remark}
 
\subsection{Symmetric approximation of curves}

In general, except in the one dimensional case (see~\cite{coroncont}) the Schwarz symmetric rearrangement
is not a continuous function (see~\cite{almgrenlieb}). To overcome this problem,
we recall a very useful and general approximation 
tool for continuous curves in Banach spaces provided 
in Jean Van Schaftingen's paper (see~\cite[Proposition 3.1]{jvsh}).

\begin{proposition}
\label{approxxres}
Let $X$ and $V$ be two Banach spaces, $S\subseteq X$, $*$ and ${\mathcal H}_*$ 
which satisfy the requirements of the
abstract symmetrization framework. Let $M$ be a metric space, $M_0$ and $M_1$ be 
disjoint closed sets of $M$ and $\gamma \in C(M,X)$. Let $H_0\in {\mathcal H}_*$
and $\gamma(M)\subset S$. Then, for every $\delta>0$, there exists a curve 
$\tilde\gamma\in C(M,X)$ such that 
$$
\|\tilde\gamma(\tau)-\gamma(\tau)^*\|_V\leq \delta,\quad\forall \tau\in M_1,
$$
$\tilde\gamma (\tau)=\gamma(\tau)^{H_1\cdots H_{[\theta]}H_\theta}$
for all $\tau\in M$, with $H_s\in {\mathcal H}_*$ for $s\geq 0$,
$\tilde\gamma(\tau)=\gamma(\tau)^{H_0}$ for $\tau\in M_0$. Here $[\theta]$
denotes the largest integer less than or equal to $\theta$.
\end{proposition}

\section{Tools from nonsmooth critical point theory}
\label{sezastratta}

\subsection{Preliminary notions and results}
In this section we consider abstract notions
and results that will be used in the proof 
of the main results. For the definitions, we refer to~\cite{cdm,dm,ioffe,katriel},
where the theory was developed. 
Let $X$ be a metric space and let $f:X\to\bar\R$
be a function. We set
\begin{equation*}
\dom{f} = \left\{u\in X:\,f(u)<+\infty\right\}
\quad\text{and}\quad
\epi{f}=\left\{(u,\xi)\in X\times\R:\,f(u)\leq\xi\right\},
\end{equation*}
and let us define the function $\gra{f}:\epi{f}\to\R$ by 
\begin{equation}
\label{defg}
\gra{f}(u,\xi)=\xi.
\end{equation}
In the following, $\epi{f}$ will be endowed with the metric
$$
d\left((u,\xi),(v,\mu)\right)=
\left(d(u,v)^2+(\xi-\mu)^2\right)^{1/2},
$$
so that the function $\gra{f}$ is Lipschitz continuous of constant $1$.
\vskip4pt
\noindent

From now on we denote with $B(u,\delta)$ the open ball of center
$u$ and of radius $\delta$. We recall the definition of the weak slope for a
continuous function.

\begin{definition}\label{defslope}
Let $X$ be a metric space, $g:X \to \R $ a continuous function,
and $u\in X$. We denote by
$|dg|(u)$ the supremum of the real numbers $ \sigma$ in
$[0,\infty)$ such that there exist $ \delta >0$
and a continuous map
$
{\mathcal H}\,:\,B(u, \delta) \times[ 0, \delta]  \to X,
$
such that, for every $v$ in $B(u,\delta) $, and for every
$t$ in $[0,\delta]$ it results
\begin{equation*}
d({\mathcal H}(v,t),v) \leq t,\qquad
g({\mathcal H}(v,t)) \leq g(v)-\sigma t.
\end{equation*}
The extended real number $|dg|(u)$ is called the weak
slope of $g$ at $u$.
\end{definition}

Now, let us recall~\cite{dgmt} a device allowing to reduce the study of continuous
or lower semi-continuous functionals to that of Lipschitz functionals. 
The following result is proved in~\cite[Proposition 2.3]{dm}.

\begin{lemma}
\label{theorgra}
Let $f:X\to\R$ be a continuous function. 
Then, for every $(u,\xi)\in\epi{f}$, we have
$$
\ws{\gra{f}}(u,\xi)=
\left\{
\begin{array}{ll}
{{\ws{f}(u)}\over{\sqrt{1+\ws{f}(u)^2}}}
& \mbox{if $f(u)=\xi$ and $\ws{f}(u)<+\infty$,} \\
1
& \mbox{if $f(u)<\xi$ or $\ws{f}(u)=+\infty$.}
\end{array}
\right.
$$
\end{lemma}

On the basis of the previous result, one can define the weak slope 
of a  lower semi-continuous function $f$
by using $|d\gra{f}|(u,f(u))$. More precisely, we have the following

\begin{definition}
\label{defnwslsc}
Let $f:X\to \bar\R$ be a lower semi-continuous function.
For every $u\in X$ such that $f(u)\in\R$, let
\begin{equation*}
\ws{f}(u)=
\begin{cases}
\dys \frac{\ws{\gra{f}}(u,f(u))}{{\sqrt{1-\ws{\gra{f}}(u,f(u))^2}}},
& \text{if $\ws{\gra{f}}(u,f(u))<1$}, \\
\noalign{\vskip4pt}
+\infty, & \text{if $\ws{\gra{f}}(u,f(u))=1$}.
\end{cases}
\end{equation*}
\end{definition}

The previous notions allow us to give the following
\begin{definition}
Let $X$ be a metric space and $f:X\to \R\cup\{+\infty\}$
a lower semi-continuous function.
We say that $u\in\dom{f}$ is a (lower) critical
point of $f$ if $\ws{f}(u)=0$.
We say that $c\in\R$ is a (lower) critical value of
$f$ if there exists a (lower) critical point $u\in\dom{f}$
of $f$ with $f(u)=c$.
\end{definition}

\begin{definition}\label{defps}
Let $X$ be a metric space, $f:X\to \R\cup\{+\infty\}$
a lower semi-continuous function and
let $c\in\R$. We say that $f$ satisfies
the Palais-Smale condition at level $c$ ($(PS)_c$
in short), if every sequence $(u_n)$ in $\dom{f}$ such that
$\ws{f}(u_n) \to 0$ and $f(u_n) \to c$
admits a subsequence $(u_{n_k})$ converging in $X$.
\end{definition}

In~\cite{cdm,dm} variational methods for lower semi-continuous
functionals are developed. Moreover, it is shown that
the following condition is fundamental in order to apply the abstract theory
to the study of lower semi-continuous functions
\begin{equation}
\label{keycond}
\forall (u,\xi)\in\epi{f}:\,\,\,
f(u)<\xi\,\,\,\Longrightarrow\,\,\,\ws{\gra{f}}(u,\xi)=1.
\end{equation}

\vskip5pt
\noindent
Let $\rho>0$ and assume that $\epi{f}$ is endowed with the metric
\begin{equation}
	\label{rhometric}
d_\rho\left((u,\xi),(v,\mu)\right)=
\left(d(u,v)^2+\rho^2(\xi-\mu)^2\right)^{1/2},
\end{equation}
Clearly the metric $d_\rho$ is equivalent to the metric $d$ on $\epi{f}$.
Moreover, with respect to $d_\rho$ the function $\gra{f}$ is 
Lipschitz continuous of constant $1/\rho$. 

\begin{proposition}
	\label{2metrics}
Let $f:X\to\bar\R$ be a function. Then
$$
\wsr{\gra{f}}(u,\xi)=
{{\ws{\gra{f}}(u,\xi)}\over{\sqrt{1+(\rho^2-1)\ws{\gra{f}}(u,\xi)^2}}},
$$
for every $(u,\xi)\in\epi{f}$. In particular, if $\ws{\gra{f}}(u,\xi)=1$, it follows
$\wsr{\gra{f}}(u,\xi)=\frac{1}{\rho}$.
\end{proposition}
\begin{proof}
	The proof follows the lines of the proof of~\cite[Proposition 2.3]{dm}. On the other hand, for
	the sake of completeness, we sketch the proof. Let us first prove that
$$
\wsr{\gra{f}}(u,\xi)\geq {{\ws{\gra{f}}(u,\xi)}\over{\sqrt{1+(\rho^2-1)\ws{\gra{f}}(u,\xi)^2}}},
$$
If $\ws{\gra{f}}(u,\xi)=0$, then there is nothing to prove. 
Otherwise, let $0<\sigma<\ws{\gra{f}}(u,\xi)$ and let
${\cal H}:\ball{u,\xi}{\delta} \times [0,\delta] \to \epi{f}$ be a
continuous map, according to the definition of weak slope, so that
$$
d({\cal H}_1((v,\mu),t),v)^2+|{\cal H}_2((v,\mu),t)-\mu|^2\leq t^2,\qquad
{\cal H}_2((v,\mu),t)\leq \mu-\sigma t,
$$
for all $t\in [0,\delta]$ and every $(v,\mu)\in \ball{u,\xi}{\delta}$.
Let now choose $\delta'>0$ (with $\delta'=\delta$ if $\rho\leq 1$) be such that $\delta'\leq \delta\sqrt{1+(\rho^2-1)\sigma^2}$
and consider the continuous function ${\cal K}=({\cal K}_1,{\cal K}_2):\ball{u,\xi}{\delta'}\times[0,\delta']\to \epi{f}$
defined by
\begin{align*}
{\cal K}_1((v,\mu),t) &={\cal H}_1\big((v,\mu),{t\over\sqrt{1+(\rho^2-1){\sigma}^2}}\big),\\
{\cal K}_2((v,\mu),t) &=\mu-{\sigma t\over\sqrt{1+(\rho^2-1){\sigma}^2}},
\end{align*}
Of course ${\cal K}((v,\mu),t)\in \epi{f}$ for all $t\in [0,\delta']$ and every $(v,\mu)\in \ball{u,\xi}{\delta'}$.
Moreover, we have
\begin{align*}
	& d_\rho({\cal K}((v,\mu),t),(v,\mu))^2  =d({\cal K}_1((v,\mu),t),v)^2+\rho^2|{\cal K}_2((v,\mu),t)-\mu|^2 \\
	\noalign{\vskip4pt}
	& =d\big({\cal H}_1((v,\mu),{t\over\sqrt{1+(\rho^2-1){\sigma}^2}}),v\big)^2
	+\rho^2\frac{\sigma^2t^2}{1+(\rho^2-1){\sigma}^2} \\
	& \leq {t^2\over {1+(\rho^2-1){\sigma}^2}}
	+\frac{(\rho^2-1)\sigma^2t^2}{1+(\rho^2-1){\sigma}^2}=t^2 ,	
\end{align*}
for all $t\in [0,\delta']$ and every $(v,\mu)\in \ball{u,\xi}{\delta'}$. Furthermore, we have
\begin{align*}
	& \gra{f}({\cal K}((v,\mu),t))={\cal K}_2((v,\mu),t)
	= \gra{f}(v,\mu)-{\sigma \over\sqrt{1+(\rho^2-1){\sigma}^2}}t,
\end{align*}
for all $t\in [0,\delta']$ and every $(v,\mu)\in \ball{u,\xi}{\delta'}$. In turn, we have
$$
\wsr{\gra{f}}(u,\xi)\geq \frac{\sigma}{\sqrt{1+(\rho^2-1)\sigma^2}},
$$
yielding, by the arbitrariness of $\sigma$,
$$
\wsr{\gra{f}}(u,\xi)\geq \frac{\ws{\gra{f}}(u,\xi)}{\sqrt{1+(\rho^2-1)\ws{\gra{f}}(u,\xi)^2}}.
$$
Concerning the proof of the opposite inequality,
it is sufficient to show that
$$
\ws{\gra{f}}(u,\xi)\geq {{\wsr{\gra{f}}(u,\xi)}\over{\sqrt{1-(\rho^2-1)\wsr{\gra{f}}(u,\xi)^2}}}.
$$
This can be achieved by arguing as above by considering, in place of ${\mathcal K}$, the continuous function
$\hat{\cal K}=(\hat{\cal K}_1,\hat{\cal K}_2):\ball{u,\xi}{\delta'}\times[0,\delta']\to \epi{f}$
defined by
\begin{align*}
\hat{\cal K}_1((v,\mu),t) &={\cal H}_1\big((v,\mu),{t\over\sqrt{1-(\rho^2-1){\sigma}^2}}\big),\\
\hat{\cal K}_2((v,\mu),t) &=\mu-{\sigma t\over\sqrt{1-(\rho^2-1){\sigma}^2}}.
\end{align*}
This concludes the proof.
\end{proof}

\begin{remark}\rm
The notion of weak slope for a function $f:X\to\overline{\R}$ (not even 
assumed to be lower semi-continuous) was also provided
(see~\cite[Definition 2.1]{campad}) in terms of local deformations, 
consistently with Definition~\ref{defslope} 
(see~\cite[Proposition 2.2]{campad}). Of course, the extended real number $\ws{f}(u)$ 
is independent of $\rho$ and, arguing as in~\cite[Proposition 2.3]{campad}, it is possible
to show that
$$
\wsr{\gra{f}}(u,f(u))=
{{\ws{f}(u)}\over{\sqrt{1+\rho^2\ws{f}(u)^2}}},
\qquad
\ws{\gra{f}}(u,f(u))=
{{\ws{f}(u)}\over{\sqrt{1+\ws{f}(u)^2}}},
$$
for every $u$ with $\ws{f}(u)<+\infty$. In turn, combining these equalities one immediately
obtains the assertion of Proposition~\ref{2metrics} with $\xi=f(u)$.
\end{remark}

\vskip4pt
\noindent

\subsection{The non-symmetric minimax theorem}
In the framework of the previous section we have the following nonsmooth minimax principle
(for $C^1$ functionals, see the corresponding version in~\cite{willem}).

\begin{theorem}
\label{mpconc}
Let $X$ be a complete metric space and $f:X\to\R\cup\{+\infty\}$ a
lower semi-continuous function satisfying \eqref{keycond}.
Let $\Dis$ and $\Sf$ denote the closed unit ball and the sphere in $\R^N$ 
respectively and $\Gamma_0\subset C(\Sf,X)$. Let us define 
$$
\Gamma=\big\{\gamma\in C(\Dis,X):\,\,\,\gamma|_{\Sf}\in \Gamma_0\big\}.
$$
Assume that
$$
+\infty>c=\inf_{\gamma\in\Gamma}\sup_{\tau\in \Dis} f(\gamma(\tau))
>\sup_{\gamma_0\in \Gamma_0}\sup_{\tau\in \Sf} f(\gamma_0(\tau))=a.
$$
Then, for every $\eps\in(0,(c-a)/2)$, every $\delta>0$ and $\gamma\in \Gamma$ such that
$$
\sup_{\tau\in \Dis}f(\gamma(\tau))\leq c+\eps,
$$
there exists $u\in X$ such that
\begin{equation*}
 c-2\eps\leq f(u)\leq c+2\eps,\quad
{\rm dist}\big(u,\gamma(\Dis)\cap f^{-1}([c-3\eps,c+3\eps])\big)\leq 3\delta,\quad
|df|(u)\leq 3\eps/\delta.
\end{equation*}
\end{theorem}

\begin{proof}
We divide the proof into two cases.
\vskip2pt
\noindent
{\bf Case I.} We prove the result for continuous functions $f:X\to\R$.
In this case, the assertion follows as
a direct application of the quantitative 
nonsmooth deformation theorem~\cite{corv} 
with the stronger conclusion that there exists $u\in X$ with
$c-2\eps\leq f(u)\leq c+2\eps$ and
\begin{equation}
	\label{contincaseass}
{\rm dist}\big(u,\gamma(\Dis)\cap f^{-1}([c-2\eps,c+2\eps])\big)\leq 2\delta,\quad
|df|(u)\leq \eps/\delta.
\end{equation}
In fact, if this was not the case, by applying~\cite[Theorem 2.3]{corv}
with the choice of the closed set $A=\gamma(\Dis)\cap f^{-1}([c-2\eps,c+2\eps]$,
one could find a deformation $\eta:X\times [0,1]\to X$
such that $d(\eta(u,t),u)\leq 2\delta t$ for all $u\in X$
and $t\in [0,1]$, $f(\eta(u,t))<f(u)$ for all $u\in X$
and $t\in [0,1]$ with $\eta(u,t)\neq u$ and
\begin{equation}
	\label{conclus11}
u\in A,\,\, c-\eps\leq f(u)\leq c+\eps
\,\,\,\,
\Longrightarrow 
\,\,\,\,
f(\eta(u,1))\leq c-\eps.
\end{equation}
If $\Xi:X\to [0,1]$ is a continuous function such that $\Xi(u)=0$ if $f(u)\leq a$
and $\Xi(u)=1$ if $f(u)\geq c-\eps$, considering $\tilde\gamma\in C(\Dis,X)$
defined by $\tilde\gamma(\tau)=\eta(\gamma(\tau),\Xi(\gamma(\tau)))$, it follows that
$\tilde\gamma \in\Gamma$, since for all $\tau\in \Sf$ we have $\Xi(\gamma(\tau))=0$, due to
$$
f(\gamma(\tau))\leq\sup_{\gamma_0\in \Gamma_0}\sup_{\tau\in \Sf} f(\gamma_0(\tau))=a.
$$
Given an arbitrary $\tau\in \Dis$, either $f(\gamma(\tau))<c-\eps$ and thus
$f(\tilde\gamma(\tau))\leq f(\gamma(\tau))<c-\eps$ or
$f(\gamma(\tau))\geq c-\eps$, in which case, 
by the definition of $\Xi$ and~\eqref{conclus11}, we get
$$
f(\tilde\gamma(\tau))=f(\eta(\gamma(\tau),\Xi(\gamma(\tau))))=f(\eta(\gamma(\tau),1))\leq c-\eps.
$$
Hence, we conclude that $f\circ \tilde\gamma|_{\Dis}\leq c-\eps$,
providing the desired contradiction with the definition of $c$
and concluding the proof for the case of $f:X\to\R$ continuous.
\vskip2pt
\noindent
{\bf Case II.}
We cover the general case of lower semi-continuous functions
$f:X\to\R\cup\{+\infty\}$.
We introduce the sets $\hat\Gamma_0\subset C(\Sf,\epi{f})$ and 
$\hat\Gamma\subset C(\Dis,\epi{f})$ by setting
\begin{align*}
\hat\Gamma_0&=\big\{\hat\gamma\in C(\Sf,\epi{f}):\text{$\hat\gamma=(\hat\gamma_1,\hat\gamma_2)$ with  $\hat\gamma_1\in \Gamma_0$ and
$\hat\gamma_2(\tau)\leq a$ for all $\tau\in \Sf$}\big\}, \\
\hat\Gamma&=\big\{\hat\gamma\in C(\Dis,\epi{f}):\hat\gamma|_{\Sf}\in \hat\Gamma_0\big\}.
\end{align*}
The space $\epi{f}$ is equipped with the metric $d_\rho$ defined in~\eqref{rhometric}, for 
$\rho>0$. As we prove below, $\hat\Gamma\neq\emptyset$.
Of course, by the definition of $\hat\Gamma_0$, we have
\begin{equation}
	\label{firstrelat}
\sup_{\hat\gamma\in \hat\Gamma_0}\sup_{\tau\in \Sf} \gra{f}(\hat\gamma(\tau))=a.
\end{equation}
Let us now prove that
\begin{equation}
	\label{secondrelat}
\inf_{\hat\gamma\in\hat\Gamma}\sup_{\tau\in \Dis} \gra{f}(\hat\gamma(\tau))=c.
\end{equation}
We first show that
\begin{equation}
	\label{prima}
\inf_{\hat\gamma\in\hat\Gamma}\sup_{\tau\in \Dis} \gra{f}(\hat\gamma(\tau))\leq c.
\end{equation}
In fact, given $b>c$, let $\gamma\in \Gamma$ be such that
$$
a<\alpha:=\sup_{\tau\in \Dis} f(\gamma(\tau))\leq b.
$$
Consider now the continuous function $\vartheta:\Dis\to \Dis$ defined by setting
$\vartheta(\tau)=\tau|\tau|^{-1}$ for all $\tau\in \overline{\Dis\setminus\Dis/2}$
and $\vartheta(\tau)=2\tau$ for all $\tau\in \Dis/2$, and define $\hat\gamma_1:\Dis\to X$
by setting $\hat\gamma_1(\tau)=\gamma(\vartheta(\tau))$ for all $\tau\in \Dis$. 
Furthermore, for any $M\geq M_0$ with
$$
M_0:=\max_{\tau\in\Sf}\max_{\tilde\tau\in \Dis/2}\textstyle{\frac{b-a}{|\tilde\tau-\tau|}}=2(b-a)>0,
$$
we introduce a continuous function $\hat\gamma_2:\Dis\to\R$ by setting
$$
\hat\gamma_2(\tau):=\sup\big\{f(\hat\gamma_1(\tilde\tau))-M|\tau-\tilde\tau|: \tilde\tau\in\Dis\big\}.  
$$
Of course $f(\hat\gamma_1(\tau))\leq \hat\gamma_2(\tau)$ for $\tau\in\Dis$ and, by an easy check,
$$
\max_{\tau\in\Dis}\hat\gamma_2(\tau)=\alpha.
$$
Furthermore, being $M\geq M_0$ and $f(\hat\gamma_1)|_{\overline{\Dis\setminus\Dis/2}}\leq a$, 
it is readily seen that $\hat\gamma_2(\tau)\leq a$
for all $\tau\in\Sf$. Therefore, taking into account that by construction 
$\hat\gamma_1|_{\Sf}=\gamma\circ\vartheta|_{\Sf}=\gamma|_{\Sf}\in\Gamma_0$,
it follows that $\hat\gamma=(\hat\gamma_1,\hat\gamma_2)\in\hat\Gamma$, yielding
$$
\inf_{\hat\gamma\in\hat\Gamma}\sup_{\tau\in \Dis} \gra{f}(\hat\gamma(\tau))\leq
\sup_{\tau\in \Dis} \gra{f}(\hat\gamma(\tau))=\alpha=\sup_{\tau\in \Dis} f(\gamma(\tau))\leq b,
$$
which proves~\eqref{prima} by the arbitrariness of $b$. On the contrary, given $d$ with
$$
d>\inf_{\hat\gamma\in\hat\Gamma}\sup_{\tau\in \Dis} \gra{f}(\hat\gamma(\tau)),
$$
we find $\hat\gamma=(\hat\gamma_1,\hat\gamma_2)\in \hat\Gamma$ with
$$
\sup_{\tau\in \Dis}\gra{f}(\hat\gamma(\tau))\leq d.
$$
Then, we have $\hat\gamma_1\in \Gamma$ and
$f(\hat\gamma_1(\tau))\leq \hat\gamma_2(\tau)=\gra{f}(\hat\gamma(\tau))\leq d$, for all $\tau\in \Dis$.
In particular we get $c\leq d$, yielding the desired inequality by the arbitrariness of $d$.
This concludes the proof of formula~\eqref{secondrelat}.
At this point, in light of~\eqref{firstrelat} and \eqref{secondrelat},
given $\eps\in(0,(c-a)/2)$, $\delta>0$ and $\gamma\in\Gamma$ with
$\sup_{\tau\in \Dis}f(\gamma(\tau))\leq c+\eps$, if $\hat\gamma_1$ and $\hat\gamma_2$
are defined as before, we have $\hat\gamma=(\hat\gamma_1,\hat\gamma_2)\in \hat\Gamma$ 
with $\hat\gamma_1(\Dis)=\gamma(\Dis)$ and 
$$
\sup_{\tau\in \Dis}\gra{f}(\hat\gamma(\tau))\leq c+\eps,
$$
and we can apply the theorem (cf.\ \eqref{contincaseass}) to the continuous function $\gra{f}$, yielding
the existence of a pair $(u,\lambda)\in \epi{f}$ such that $c-2\eps\leq \lambda\leq c+2\eps$ and
\begin{equation}
	\label{cisiamquasi}
{\rm dist}_\rho\big((u,\lambda),\hat\gamma(\Dis)\cap \gra{f}^{-1}([c-2\eps,c+2\eps])\big)\leq 2\delta,\quad
|d_\rho\gra{f}|(u,\lambda)\leq \eps/\delta.
\end{equation}
Now, by choosing $\rho:=\frac{2\sqrt{2}\delta}{3\eps}$ for the metric in $\epi{f}$, we have
$$
|d_\rho\gra{f}|(u,\lambda)\leq \eps/\delta<\frac{1}{\rho}.
$$ 
Therefore, by virtue of Proposition~\ref{2metrics} and in light of condition~\eqref{keycond},
we deduce that $\lambda=f(u)$, which yields
\begin{equation}
	\label{finllassss}
c-2\eps\leq f(u)\leq c+2\eps,\qquad
{\rm dist}\big(u,\gamma(\Dis)\cap f^{-1}([c-3\eps,c+3\eps])\big)\leq 3\delta.
\end{equation}
Concerning the second assertion, observe that from the first inequality of~\eqref{cisiamquasi}, 
replacing $\delta$ with a slightly larger $\delta$ if necessary, there exists $\tau\in\Dis$ such that
$$
d(u,\hat\gamma_1(\tau))\leq 2\delta,\qquad c-2\eps\leq \hat\gamma_2(\tau)\leq c+2\eps.
$$ 
Now, by continuity, there exists $\delta'>0$ such that 
$$
\forall\tilde\tau\in\Dis:\,\, |\tilde\tau-\tau|\leq\delta'\,\,\,\Rightarrow\,\,\,
d(\hat\gamma_1(\tilde\tau),\hat\gamma_1(\tau))\leq \delta,\quad
c-3\eps\leq \hat\gamma_2(\tilde\tau)\leq c+3\eps.
$$
Observe now that, for any given $\mu\in\R$, it follows
$$
 f(\hat\gamma_1)|_{\{\tilde\tau\in\Dis:\,|\tilde\tau-\tau|\leq\delta'\}}\leq\mu
\quad\Longrightarrow\quad
\hat\gamma_2(\tau)\leq \mu,
$$
if $M\geq \max\{M_0,\frac{c+\eps-\mu}{\delta'}\}$
in the definition of $\hat\gamma_2$. In fact, it holds
\begin{align*}
&\forall\tilde\tau\in\Dis:\,\, |\tilde\tau-\tau|\leq\delta'\,\,\,\Rightarrow\,\,\, 
f(\hat\gamma_1(\tilde\tau))-M|\tau-\tilde\tau|\leq \mu, \\
&\forall\tilde\tau\in\Dis:\,\, |\tilde\tau-\tau|>\delta'\,\,\,\Rightarrow\,\,\, 
f(\hat\gamma_1(\tilde\tau))-M|\tau-\tilde\tau|\leq c+\eps-M\delta'\leq\mu.
\end{align*}
Hence, since $\hat\gamma_2(\tau)> c-3\eps$, if $M\geq \max\{M_0,\frac{4\eps}{\delta'}\}$ in the definition of $\hat\gamma_2$,
we have
$$
\exists\tilde\tau\in\Dis:\,\,\, |\tilde\tau-\tau|\leq\delta'\,\,\,\,\text{and}\,\,\,\,
c-3\eps<f(\hat\gamma_1(\tilde\tau))\leq \hat\gamma_2(\tilde\tau)\leq c+3\eps.
$$
Since $\hat\gamma_1(\tilde\tau)\in \gamma(\Dis)\cap f^{-1}([c-3\eps,c+3\eps])$, we obtain 
\begin{equation*}
{\rm dist}\big(u,\gamma(\Dis)\cap f^{-1}([c-3\eps,c+3\eps])\big)\leq 
d(u,\hat\gamma_1(\tilde\tau))\leq  d(u,\hat\gamma_1(\tau))+d(\hat\gamma_1(\tau),\hat\gamma_1(\tilde\tau))\leq 3\delta.
\end{equation*}
Finally, by virtue of Proposition~\ref{2metrics}, it holds
$$
\ws{f}(u)=\frac{\ws{\gra{f}}(u,f(u))}{{\sqrt{1-\ws{\gra{f}}(u,f(u))^2}}}
=\frac{\wsr{\gra{f}}(u,f(u))}{{\sqrt{1-\rho^2\wsr{\gra{f}}(u,f(u))^2}}}
\leq \frac{\eps/\delta}{{\sqrt{1-\rho^2\eps^2/\delta^2}}}= 3\eps/\delta.
$$
This concludes the proof.
\end{proof}

\subsection{The symmetric minimax theorem}

The main abstract tool of the paper is a symmetric version of Theorem~\ref{mpconc},
namely the following

\begin{theorem}
\label{mpconc-symm}
Let $X$ and $V$ be two Banach spaces, $S\subset X$, $*$ and ${\mathcal H}_*$ 
satisfying the requirements of the abstract symmetrization framework. 
Let $f:X\to\R\cup\{+\infty\}$ a
lower semi-continuous function satisfying $\refe{keycond}$.
Let $\Dis$ and $\Sf$ denote the closed unit ball and the sphere in $\R^N$ 
respectively and $\Gamma_0\subset C(\Sf,X)$. Let us define 
$$
\Gamma=\big\{\gamma\in C(\Dis,X):\,\,\,\gamma|_{\Sf}\in \Gamma_0\big\}.
$$
Assume that
$$
+\infty>c=\inf_{\gamma\in\Gamma}\sup_{\tau\in \Dis} f(\gamma(\tau))
>\sup_{\gamma_0\in \Gamma_0}\sup_{\tau\in \Sf} f(\gamma_0(\tau))=a,
$$
and that
$$
\forall {\mathcal H}_*,\,\, \forall u\in S:\quad
f(u^H)\leq f(u).
$$
Then, for every $\eps\in(0,(c-a)/3)$, every $\delta>0$ and $\gamma\in \Gamma$ such that
$$
\sup_{\tau\in \Dis}f(\gamma(\tau))\leq c+\eps,\quad \gamma(\Dis)\subset S,\quad
\text{$\gamma|_{\Sf}^{H_0}\in\Gamma_0$ for some $H_0\in {\mathcal H}_*$},
$$
there exists $u\in X$ such that
\begin{equation}
\label{conclusmmm}
 c-2\eps\leq f(u)\leq c+2\eps,\quad
|df|(u)\leq 3\eps/\delta,\quad
\|u-u^*\|_V\leq 3(2K+1)\delta,
\end{equation}
being $K$ the norm of the embedding map $i:X\to V$.
\end{theorem}
\begin{proof}
Let	$\eps\in(0,(c-a)/3)$, $\delta>0$ and $\gamma\in \Gamma$ satisfying the assumptions.
Moreover, let $\vartheta:\Dis\to\Dis$ be the continuous function introduced in the proof of Theorem~\ref{mpconc},
and consider the function $\eta:\Dis\to X$, defined as $\eta(\tau):=\gamma(\vartheta(\tau))$ for all $\tau\in\Dis$.
Then $\eta\in\Gamma$, we have $\eta(\Dis)=\gamma(\vartheta(\Dis))=\gamma(\Dis)\subset S$ and, setting 
$$
M_1:=\overline{(f\circ \eta)^{-1}([c-3\eps,c+\eps])},
$$ 
$M_1\subset \Dis$ is of course closed and $M_1\cap \Sf=\emptyset$. 
In fact, assume by contradiction that this is not the case and let
$\tau\in M_1\cap \Sf$. Then
$$
\tau\in \Sf,\,\,\,\tau=\lim_j\tau_j,\quad
c-3\eps\leq f(\gamma(\vartheta(\tau_j)))\leq c+\eps,\,\,\,\text{for all $j\geq 1$}.
$$
In particular, $\tau_j\in \overline{\Dis\setminus\Dis/2}$ eventually for $j\geq 1$, so that
$\vartheta(\tau_j)\in\Sf$ eventually for $j\geq 1$. Therefore, for such $j\geq 1$, we obtain
$$
c-3\eps\leq f(\gamma(\vartheta(\tau_j)))\leq \sup_{\tau\in\Sf} f(\gamma(\tau))\leq
\sup_{\gamma_0\in\Gamma_0}\sup_{\tau\in\Sf} f(\gamma_0(\tau))=a<c-3\eps
$$
yielding the desired contradiction.
Now, from Proposition~\ref{approxxres} (applied with the choice $M=\Dis$ and $M_0=\Sf$)
there exists a curve $\tilde\eta\in C(\Dis,X)$ with $\tilde\eta|_{\Sf}=\eta|_{\Sf}^{H_0}=\gamma|_{\Sf}^{H_0}\in\Gamma_0$ 
for the polarizer $H_0$ (so that $\tilde\eta\in\Gamma$) such that 
$\|\tilde\eta(\tau)-\eta(\tau)^*\|_V\leq 3\delta$, for all $\tau\in M_1$. Notice that,  
by construction, $\tilde\eta(\tau)^*=\eta(\tau)^*$ and $f(\tilde\eta(\tau))
\leq f(\eta(\tau))$ for every $\tau\in \Dis$, 
as $\tilde\eta$ is built from $\eta$ through polarizations. Hence, we obtain
$$
\sup_{\tau\in \Dis}f(\tilde\eta(\tau))\leq \sup_{\tau\in \Dis}f(\eta(\tau))=\sup_{\tau\in \Dis}f(\gamma(\tau))\leq c+\eps.
$$
By applying Theorem~\ref{mpconc} to $\tilde\eta$ and since 
$$
\tilde\eta(\Dis)\cap f^{-1}([c-3\eps,c+3\eps])\subset \tilde\eta(M_1),
$$
there exists $u\in X$ such that
${\rm dist}(u,\tilde\eta(M_1))\leq 3\delta$ and
the first two inequalities in the above formula~\eqref{conclusmmm} hold. The last assertion in~\eqref{conclusmmm}
just follows by adding and subtracting $\tilde \eta(\tau)$
and $\eta(\tau)^*$ with $\tau\in M_1$, as in~\cite[proof of Theorem 3.2]{jvsh}, namely
\begin{align*}
\|u-u^*\|_V &\leq\inf_{\tau\in M_1}\big[
\|u-\tilde \eta(\tau)\|_V+
\|\tilde \eta(\tau)-\eta(\tau)^*\|_V+
\|\eta(\tau)^*-u^*\|_V \big] \\
&\leq\inf_{\tau\in M_1}\big[
2\|u-\tilde \eta(\tau)\|_V+
\|\tilde \eta(\tau)-\eta(\tau)^*\|_V\big]\leq 3(2K+1)\delta.
\end{align*}
This concludes the proof.
\end{proof}

\begin{remark}\rm
	\label{conccomp-r}
Let $X$ and $V$ be two Banach spaces such that $X$ is continuously embedded in $V$
and let $S\subset X$. We consider a symmetrization map $*:S\to V$ which satisfies the requirements
of the abstract symmetrization framework.	
Theorem~\ref{mpconc-symm} provides, in some sense, a useful alternative to concentration
compactness. In fact, for a broad range of lower semi-continuous functionals $f:X\to\R\cup\{+\infty\}$
possessing a mountain pass geometry, Theorem~\ref{mpconc-symm} yields
a sequence of functions $(u_h)\subset X$ such that, as $h\to\infty$,
\begin{equation}
\label{conclusmmm-conc}
f(u_h)\to c,\quad
|df|(u_h)\to 0,\quad
\|u_h-u^*_h\|_V\to 0.
\end{equation}
It is often the case that the first two limits yield the boundedness of $(u_h)$ in $X$,
so that $u_h\to u$ weakly in $X$ for some $u\in X$
and that the symmetric sequence $(u^*_h)\subset X_r$ converges strongly, 
up to a subsequence, to some $v\in X_r$ in some subspace $V'$ with 
$V\subset V'$ with continuous injection $i:V\to V'$. In particular,
$$
u_h^*\to v\quad\text{in $V'$ as $h\to\infty$},\qquad \|u_h-u^*_h\|_{V'}\leq C\|u_h-u^*_h\|_V\to 0,
$$
which yields 
\begin{equation}
	\label{convstro}
\text{$u_h\to u$ {\em weakly} in $X$ and {\em strongly} in $V'$.}
\end{equation}
This conclusion is often
sufficient in order to prove, after some work, that the Palais-Smale sequence $(u_h)$
converges to $u$ strongly in $X$. As a concrete functional framework one can think, for instance,
to the case (here $\Omega$ can be the whole $\R^N$) where
$$
X=W^{1,p}_0(\Omega),\quad
V=L^p\cap L^{p^*}(\Omega),\quad 
V'=L^m(\Omega),\,\,\, p<m<p^*. 
$$
Therefore, if $(u_h)$ is bounded in $W^{1,p}_0(\Omega)$, the sequence 
$(u_h^*)$ is bounded in $W^{1,p}_0(\Omega)$ too by the Polya-Szeg\"o
inequality and compact in $L^m(\Omega)$ with $p<m<p^*$ in light
of~\cite[Theorem A.I', p.341]{BL1}. Finally, the injection $i:L^p\cap L^{p^*}(\Omega)\to L^m(\Omega)$
is, of course, continuous. For an application of 
conclusion~\eqref{convstro} in the case $p=2$, $\Omega=\R^N$ and $f\in C^1(H^1(\R^N),\R)$, 
see~\cite[Theorem 4.5]{jvsh}.
\end{remark}

\begin{remark}\rm
	As pointed out in~\cite{jvsh} the condition that 
	$\gamma|_{\Sf}^{H_0}\in\Gamma_0$ for some polarizer $H_0\in {\mathcal H}_*$
	imposes a minimality condition on the energy levels on which one can guarantee
	the symmetry of critical points.
\end{remark}

\section{Proof of Theorem~\ref{mainth}}

\subsection{Some preliminary Lemmas}
Given a fixed function $u$ in $\hsob$, we
define the following subspace of $\hsob$ 
\begin{equation}\label{defwu}
W_u=\big\{v\in\hsob:\, \jxi{u}\cdot \nabla v\in\elle1
\,\,\,
\text{and}
\,\,\,
\js{u}v\in \elle1\big\}.
\end{equation}
The space $W_u$ is dense in $\hsob$. It was originally introduced in~\cite{degzan}
and subsequently used also throughout~\cite{pelsqu}.
We give the definition of generalized solution.
\begin{definition}\label{defsol}
We say that $u$ is a generalized solution to~\eqref{theproblem}
if $u\in \hsob$ and it results $\jxi{u}\cdot\nabla u\in\elle1$, $\js{u}u\in \elle1$ and
\begin{equation*}
\dys \into \jxi{u}\cdot \nabla v dx +\into \js{u}v dx=\int_{B_1} g(|x|,u)v dx, \qquad
\forall\, v\in W_u.
\end{equation*}
\end{definition}

We recall some preliminary results. 

\begin{lemma}
\label{legamesol}
Assume that conditions~\eqref{j1}-\eqref{j6} hold. If $u\in{\rm dom}(f)$
is a critical point of $f$, namely $|df|(u)=0$, then $u$ is a generalized solution to
\begin{equation*}
\begin{cases}
-\,\dvg(j_\xi(u,|\nabla u|)) +j_s(u,|\nabla u|)=g(|x|,u), & \text{in $B_1$}, \\
\quad u=0,  & \text{on $\partial B_1$}.
\end{cases}
\end{equation*}
Furthermore, if $j_\xi(u,|\nabla u|)\cdot\nabla u\in L^1(B_1)$, then
$u$ is a distributional solution.
\end{lemma}
\begin{proof}
Combine~\cite[Proposition 6.4 and Theorem 4.10]{pelsqu}.
\end{proof}

\begin{lemma}
\label{penduno}
Assume that conditions~\eqref{j1}-\eqref{j6} hold.
Then, for every $(u,\xi)\in\epi{J}$ with $f(u)<\xi$, there holds
$|d{\mathcal G}_f|(u,\xi)=1$.
\end{lemma}
\begin{proof}
See~\cite[Theorem 6.1]{pelsqu}.
\end{proof}

\begin{lemma}
\label{mtpgeo}
Assume that conditions~\eqref{j1}-\eqref{j6} hold.
Then, there exists $e\in H^1_0(B_1)$ such that $f(e)<0$ and $\rho,\sigma>0$
such that $f(u)\geq\sigma$ for all $u\in H^1_0(B_1)$ with 
$\|u\|_{H^1_0}=\rho$.
\end{lemma}
\begin{proof}
See the beginning of the proof of~\cite[Theorem 2.3]{pelsqu}.
\end{proof}

\begin{lemma}
\label{fps}
Assume that conditions~\eqref{j1}-\eqref{j6} hold.
Then the functional $f$ satisfies the $(PS)_c$
condition at every level $c\in\R$.
\end{lemma}
\begin{proof}
See~\cite[Theorem 6.9]{pelsqu}.
\end{proof}

\begin{lemma}
\label{concretelem}
Let $u\in H^1_0(B_1,\R^+)$ and let $H$ be a given half-space. Then
 \begin{equation}
	\label{polar-quasil}
\int_{B_1}j(u,|\nabla u|)dx=\int_{B_1}j(u^H,|\nabla u^H|)dx,
\end{equation}
	provided that $0\in H$ and that both integrals are finite. Furthermore, under~\eqref{gg5}, 
\begin{equation*}
\int_{B_1}G(|x|,u)dx\leq \int_{B_1}G(|x|,u^H)dx.
\end{equation*}
\end{lemma}
\begin{proof}
See~\cite[Lemma 2.5]{HSq} and \cite[Theorem 6.4]{HS}, respectively. 
Concerning~\cite[Lemma 2.5]{HSq}, statement~\eqref{polar-quasil} is 
provided for functions $\tilde u:\R^N\to\R^+$, that is
\begin{equation}
\label{entirepolarizz}
\int_{\R^N}j(\tilde u,|\nabla \tilde u|)dx=\int_{\R^N}j(\tilde u^H,|\nabla \tilde u^H|)dx.
\end{equation}
On the other hand, given a function $u:B_1\to\R^+$, if $\tilde u:\R^N\to \R^+$ is the extension
of $u$ by zero outside $B_1$, we have $\tilde u^H|_{\R^N\setminus B_1}=0$.
In fact, if $x\in (\R^N\setminus B_1)\cap H$, then 
$\tilde u^H(x)=\max\{\tilde u(x),\tilde u(x_H)\}=0$, being $x,x_H\in \R^N\setminus B_1$ (due to $0\in H$). If, instead,
$x\in (\R^N\setminus B_1)\cap (\R^N\setminus H)$, then 
$\tilde u^H(x)=\min\{\tilde u(x),\tilde u(x_H)\}=\min\{0,\tilde u(x_H)\}=0$, being $\tilde u\geq 0$.
The desired conclusion~\eqref{polar-quasil} then follows from~\eqref{entirepolarizz},
being $j(s,0)=0$. 
\end{proof}

\vskip4pt
\noindent
\subsection{Proof of Theorem~\ref{mainth} concluded.}
In view of Lemma~\ref{concretelem} we have $f(u^H)\leq f(u)$, for every $u\in H^1_0(B_1,\R^+)$
and all polarizer $H\in {\mathcal H}_*$. This holds for all $u\in H^1_0(B_1)$ as well. 
In fact, notice that, since for sign changing 
functions $u^H:=|u|^H$, taking into account assumptions~\eqref{opposite} and~\eqref{gg6},
we obtain that
$$
f(u^H)=f(|u|^H)\leq f(|u|)\leq f(u), \quad\text{for all $u\in H^1_0(B_1)$
and $H\in {\mathcal H}_*$}.
$$
By virtue of Lemma~\ref{penduno}, we are allowed to
apply the abstract symmetric minimax Theorem~\ref{mpconc-symm} to the lower 
semi-continuous functional $f:X\to\R\cup\{+\infty\}$ by choosing $X=S=H^1_0(B_1)$, 
$V=L^2\cap L^{2^*}(B_1)$, $\Dis=[0,1]$, $\Sf=\{0,1\}$,
\begin{equation*}
\Gamma=\{\gamma\in C([0,1],H^1_0(B_1)): \gamma|_{\{0,1\}}\in \Gamma_0\}
\end{equation*}
and $\Gamma_0=\{\text{$0,e$: $e\in H^1_0(B_1)$ is such that $f(e)<0$}\}$. 
If follows that $\Gamma\not=\emptyset$ in light of Lemma~\ref{mtpgeo}, also yielding $c>0=a$ by definition 
of $a$ and $c$ in Theorem~\ref{mpconc-symm}. Of course,
it holds $f(0^H)=f(0)=0$ and $f(e^H)\leq f(e)<0$, for any polarizer $H$,
so that $\gamma(0)^H,\gamma(1)^H\in\Gamma_0$, for any $\gamma\in \Gamma$.  
Moreover, take $\eps=\eps_h=1/h^2$, $\delta=\delta_h=1/h$, 
$\gamma=\gamma_h\in C([0,1],H^1_0(B_1))$ such that
$$
\sup_{\tau\in [0,1]} f(\gamma_h(\tau))\leq c+\frac{1}{h^2}.
$$
Hence, Theorem~\ref{mpconc-symm} yields a sequence $(u_h)\subset H^1_0(B_1)$ such that
\begin{equation*}
 c-\frac{2}{h^2}\leq f(u_h)\leq c+\frac{2}{h^2},\quad
|df|(u_h)\leq \frac{8}{h},\quad
\|u_h-u^*_h\|_{L^2(B_1)}\leq \frac{2(2K+1)}{h}.
\end{equation*}
In particular, $(u_h)$ is a Palais-Smale sequence at level $c$. By means of 
Lemma~\ref{fps}, up to a subsequence, $(u_h)$ strongly converges
in $H^1_0(B_1)$ to some $\hat u\in H^1_0(B_1)$ with $f(\hat u)=c>0$
(hence $\hat u$ is nontrivial) and $|df|(\hat u)=0$. In light of 
Lemma~\ref{legamesol}, it follows that $\hat u$ is a generalized solution
of the problem. Taking into account the growth condition~\eqref{gg1} on $g$,
by virtue of~\cite[Theorem 7.1(b)]{pelsqu} it follows that $\hat u\in L^\infty(\Omega)$.
Now, by combining assumptions~\eqref{j1} and~\eqref{j2}, it holds
$$
|j_\xi(s,|\xi|)|\leq 4\alpha(|s|)|\xi|,
$$
for every $s\in\R$ and all $\xi\in\R^N$ (cf.\ \cite[Remark 4.1]{pelsqu}).
Then, again by Lemma~\ref{legamesol}, it follows that $\hat u$ is a distributional
solution, being
$$
\int_{B_1} |j_\xi(\hat u,|\nabla \hat u|)\cdot\nabla \hat u|dx\leq \int_{B_1} 4\alpha(\hat u)|\nabla \hat u|^2dx
\leq 4\alpha(M)\int_{B_1}|\nabla \hat u|^2dx<\infty,
$$
where $M={\rm esssup}_{B_1} |u|$. Finally, from $\|u_h-u^*_h\|_{L^2(B_1)}\to 0$ 
and $u_h\to \hat u$ in $L^2(B_1)$ as $h\to\infty$, we get
$u_h^*\to \hat u^*$ and $u_h^*\to \hat u$ in $L^2(B_1)$, so that $\hat u=\hat u^*$ by uniqueness
of the limit. This concludes the proof.  
\qed

\vskip20pt
\noindent
{\bf Acknowledgment.} The author wishes to thank Marco Degiovanni for a very helpful discussion
about the proof of Theorem~\ref{mpconc}.

\bigskip

\end{document}